\documentclass[10pt]{article}

\usepackage{graphicx}
\usepackage{epstopdf}
\usepackage{float}	
\usepackage{amsthm}
\usepackage{amssymb}
\usepackage{amsmath}
\usepackage{hyperref}
\usepackage[all]{xy}
\usepackage[total={4.5in,7.2in}]{geometry}
\usepackage{eucal}
\usepackage{enumerate}
\usepackage{paralist}
\usepackage{caption}
\usepackage[labelformat=simple]{subcaption}

\newcommand{\NI}{\operatorname{NI}(w)}
\newcommand{\NIin}[1]{\operatorname{NI}(#1)}
\newcommand{\R}{\mathcal{R}}

\newcommand{\C}{\mathcal{C}}
\newcommand{\M}{\mathcal{M}}

\newtheorem{thm}{Theorem}
\newtheorem{lem}{Lemma}
\newtheorem{crl}{Corollary}

\theoremstyle{remark}
\newtheorem{rem}{Remark}

\theoremstyle{definition}
\newtheorem{dfn}{Definition}
\newtheorem{exm}{Example}

\begin{document}

\title{Reductions on Double Occurrence Words}
\author{Ryan Arredondo\footnote{Email: \href{mailto: rarredon@mail.usf.edu}{rarredon@mail.usf.edu}} \\
				Department of Mathematics and Statistics \\
				University of South Florida, Tampa FL 33620, USA
				}
\date{}
\maketitle

\begin{abstract}
In the present paper we consider biologically motivated reduction operations on double occurrence words.
Then we define the nesting index of a double occurrence word to be 
the least number of reduction operations it takes for a word to be reduced to the empty word.
We use chord diagrams and circle graphs as tools to study the nesting index of double occurrence words.
\end{abstract}

{\small \textbf{Keywords:} 
Ciliate biology, Gauss codes, double occurrence word, chord diagram, circle graph
\section{Introduction}
Certain 4-valent rigid vertex graphs, called assembly graphs, have been 
used to model the genome rearrangement processes that occur in species of single-celled organisms called ciliates, for example, in \cite{angeleska2007}.
A particular class of assembly graphs can be represented by double occurrence words, also known as Gauss codes.

In the following sections of the paper we define double occurrence words of a certain form 
which relate to patterns observed \cite{prescott2000} in the scrambled genomes of the ciliate species, \emph{Oxytricha}.
We use the double occurrence words of a specific form to define reduction operations for double occurrence words in general.
In turn we define the nesting index of a double occurrence word to be 
the least number of reduction operations it takes for a word to be reduced to the empty word.
We briefly discuss the computation of the nesting index 
and we provide a table with the counts of all double occurrence words with  nesting index from $1$ to $10$ and size from $1$ to $9$.
We use the table to propose a conjecture on the minimum number of letters needed  
to construct a double occurrence word with nesting index $n \in \mathbb{N}$.

We continue our study of the nesting index with the notions of chord diagrams and circle graphs
which can be useful tools, for example in \cite{chord}, when working with double occurrence words.
In particular, we give several results that relate the nesting index of a double occurrence word to its chord diagram.
We go on to present examples of words which have isomorphic circle graphs but arbitrarily large differences in nesting indices.
We conclude the paper with some open questions involving the nesting index and circle graphs.

\section{Preliminaries}

A \emph{graph} $G = (V,E)$ is a pair consisting of a set of vertices $V$ and a set of edges $E$
where the two endpoints of an edge in $E$ are vertices in $V$.
We allow for multiple edges to be associated with a single pair of vertices;
this is sometimes referred to as a multigraph.
If $e$ is an edge and $v$ is an endpoint of $e$, then $e$ is said to be \emph{incident} to $v$.
The number of edges incident to a vertex $v$ is called the \emph{degree} of $v$.
By convention, a \emph{loop}, defined as an edge with one endpoint, contributes 2 to the degree of a vertex. 
A vertex is called \emph{rigid} if all of its incident edges are fixed in a cyclic order. 
An \emph{assembly graph} is a finite graph in which all vertices are rigid and have degree 1 or 4.
Figure~\ref{fig:assemblies} shows some examples of assembly graphs. 
A vertex with degree 1 is called an \emph{endpoint}.
In Figure~\ref{fig:assemblygraph}, $v_0$ and $v_3$ are endpoints.
In the remainder of this paper we assume that an assembly graph has endpoints, unless otherwise stated.
If $v$ is a rigid 4-valent vertex with incident edges in the cyclic arrangement $(e_1,e_2, e_3, e_4)$,
then $e_2$ and $e_4$ are called \emph{neighbors} of $e_1$ and $e_3$.
In the case that $e_2$ is a loop and $e_2 = e_3$, then we have that $e_2$ is both a neighbor and not a neighbor of $e_1$.
In Figure~\ref{fig:assemblygraph} $e_1$ has neighbors $e_3$ and $e_4$
and in Figure~\ref{fig:nonsimpleassembly} $e_3$ has neighbors $e_4$ and $e_5$.

\begin{figure}
	\centering
	\begin{subfigure}[h]{0.41\textwidth}
		\centering
			\includegraphics[scale=0.43]{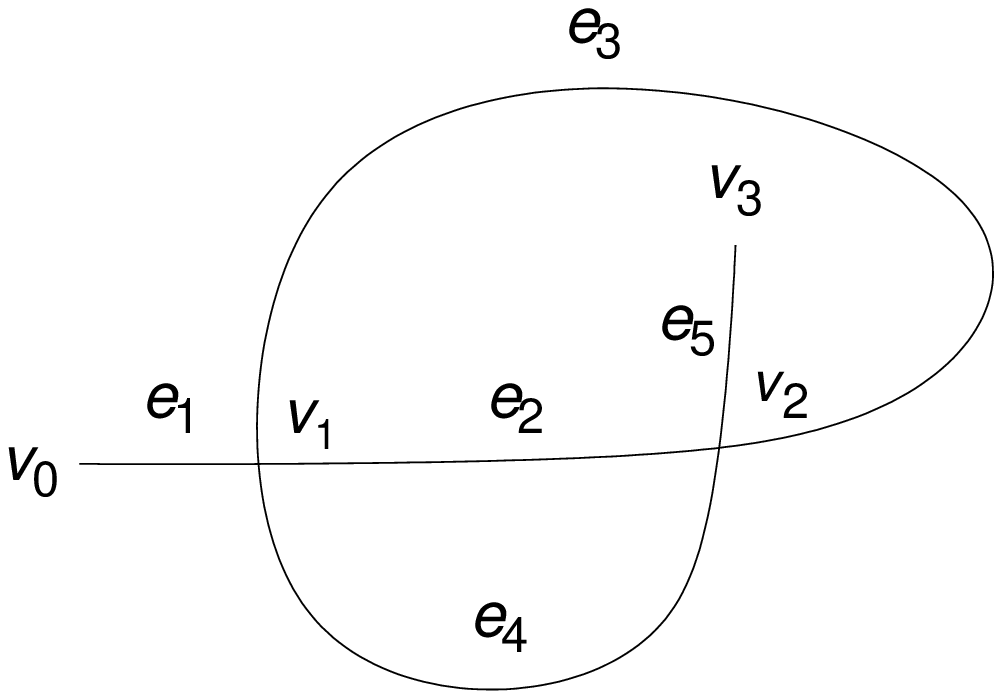}
		\caption{Simple assembly graph}
 		\label{fig:assemblygraph}
	\end{subfigure}
	\qquad
	\begin{subfigure}[h]{0.45\textwidth}
		\centering
		\hspace*{-16pt} \vspace*{4pt}
		\includegraphics[scale=0.65]{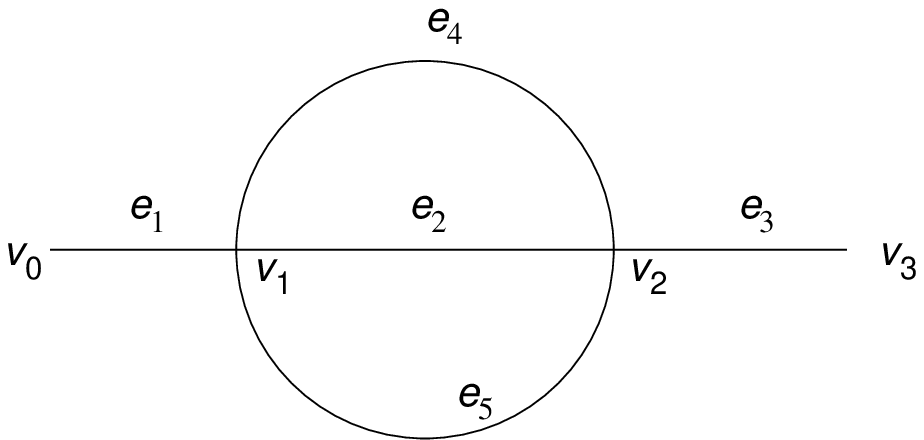}
		\caption{Non-simple assembly graph}
		\label{fig:nonsimpleassembly}
	\end{subfigure}
	\caption{Examples of assembly graphs}
	\label{fig:assemblies}
\end{figure}

Assembly graphs are of particular interest because of their recent use 
to model genome rearrangement in ciliates.
Ciliates are unicellular organisms which contain two types of nuclei,
the germline (micronuclear) and somatic (macronuclear). 
The micronucleus contains segments of DNA found in the macronucleus but often in a permuted order
and separated by non-coding DNA.
During sexual reproduction the micronuclear genome undergoes massive elimination of non-coding DNA
and rearrangement to obtain a new macronucleus.
We refer the reader to \cite{prescott2000} for a more thorough treatment of the biological background
and we recommend \cite{angeleska2007} for more on the assembly graph model.

For an assembly graph $\Gamma$ with endpoints $v_0$ and $v_n$, 
a \emph{transverse path} is a sequence
$\gamma = (v_0, e_1, v_1, e_2, \ldots, e_n, v_n)$ satisfying:
\begin{inparaenum}[(1)]
	\item $(v_0, \ldots, v_n)$ is a sequence of a subset of vertices of $\Gamma$ with possible repetition of the same vertex at most twice,
	\item $\{e_1, \ldots, e_n\}$ is a set of distinct edges such that $e_i$ is incident to $v_{i-1}$ and $v_{i}$ for $i = 2, \ldots, n$, and
	\item $e_i$ is not a neighbor of $e_{i-1}$ with respect to the rigid vertex $v_{i-1}$, for $i = 2, \ldots, n$.
\end{inparaenum}
An assembly graph $\Gamma$ is called \emph{simple} if there is a transverse \emph{Eulerian} path in $\Gamma$,
 meaning there is a transverse path that contains every edge from $\Gamma$ exactly once.
 The assembly graph in Figure~\ref{fig:assemblygraph} is simple with endpoints $v_0$ and $v_3$, while
 the assembly graph in Figure~\ref{fig:nonsimpleassembly} is non-simple with two transverse components; 
 one without endpoints and the other with endpoints $v_0$ and $v_3$.
In the remainder of this paper all assembly graphs are assumed to be simple, unless otherwise stated.

We now establish a convention for representing simple assembly graphs by words.
A \emph{double occurrence word} $w$ is a word containing symbols (or \emph{letters}) from a finite alphabet such that every symbol in $w$ appears exactly twice.
Let $\Gamma$ be a simple assembly graph with vertices $v_1,\ldots v_n$.
Given a transverse Eulerian path of $\Gamma$, $\gamma = (v_{i_0}, e_1, v_{i_1}, \ldots, e_{2n+1}, v_{i_{2n+1}})$ for $i_k \in \{1,\ldots,n\}$,
note that all vertices except endpoints, $v_{i_0}$ and $v_{i_{2n+1}}$, are visited exactly twice.
Thus, we can represent $\Gamma$ by the double occurrence word $v_{i_1} v_{i_2} \cdots v_{i_{2n}}$.
For example, the graph in Figure~\ref{fig:assemblygraph} has transverse Eulerian path $(v_0, e_1, v_1, e_2, v_2, e_3, v_1, e_4, v_2, e_5, v_3)$
and so we can represent the graph by the double occurrence word $v_1 v_2 v_1 v_2$.

It will sometimes be convenient to label double occurrence words in a conventional manner.
Let $w_1$ be a word over the alphabet $\Sigma_1$ and $w_2$ a word over the alphabet $\Sigma_2$ such that $|w_1| = n = |w_2|$.
Then we say $w_2 = b_1b_2 \cdots b_n$ is a \emph{relabeling} of $w_1 = a_1a_2 \cdots a_n$
when $a_i = a_j$ if and only if $b_i = b_j$ for all $1 \leq i \leq j \leq n$.
A word $w$ over a finite alphabet $\Sigma = \{1,2, \ldots, n\}$ is in \emph{ascending order} 
if the left-most symbol is $1$ and every other symbol in $w$ is at most 1 value greater than any symbol appearing to the left of it.
The assembly graph in Figure~\ref{fig:assemblygraph} can be represented by the double occurrence word $v_1v_2v_1v_2$ 
or in ascending order by $1212$.
We use $w^{asc}$ to denote the unique relabeling of a word $w$ such that $w^{asc}$ is in ascending order.

A double occurrence word $w$ with $n$ distinct symbols has \emph{size} $n$  and  \emph{length} $|w| = 2n$. 
We use $\epsilon$ to denote the \emph{empty word}, a word containing no symbols. 
We say $u$ is a \emph{subword} of a word $w$, written $u \sqsubseteq w$, if we can write $w = suv$, 
where $s$, $u$, and $v$ are also words (possibly empty). 
Two words $w_1$ and $w_2$ are said to be \emph{disjoint} if they have no letter in common.
If $w = a_1a_2 \cdots a_n$, then $w^R = a_n \cdots a_2 a_1$ is called the \emph{reverse} of $w$.
Two words $w_1$ and $w_2$ over an alphabet $\Sigma$ are said to be \emph{reverse equivalent} if $w_1 = w_2$ or $w_1 = w_2^R$.

\section{Reductions on double occurrence words}

In the present section we introduce notation which will be useful for defining reduction operations on double occurrence words.
We then introduce the notions of a repeat word and a return word.
We use these words to define the reduction operations on double occurrence words. 
The reduction operations will be used to define the nesting index of a double occurrence word.

\subsection{Reduction notation}

	\begin{dfn}
		If $w = w_1vw_2$ where $w$ and $v$ are both double occurrence words, 
		then $w - v = w_1w_2$ is called the \emph{subword removal} of $v$ from $w$.
	\end{dfn}
	
	\begin{dfn}
		If $D = \{v_1, v_2, \ldots, v_n\}$ is a set containing disjoint double occurrence subwords of $w$,
		and $\mathcal{O} = (v_1,\ldots,v_n)$ is an ordering of $D$, then
		we use $w - D_{\mathcal{O}}$ to mean $((\cdots((w-v_1) - v_2)\cdots) - v_n)$.
		\label{def:set-removal}
	\end{dfn}
	
	\begin{rem}
		If $D$ is a set of disjoint double occurrence subwords of $w$
		and $\mathcal{O}$ and $\mathcal{O'}$ are two orderings of $D$,
		then $w-D_{\mathcal{O}} = w-D_{\mathcal{O'}}$ and hence, we will write $w - D$.
	\end{rem}
			
	\begin{dfn}
		If $w = w_1 a w_2 a w_3$ is a double occurrence word and $a \in \Sigma$,
		then $w - a = w_1w_2w_3$ is called the \emph{letter removal} of $a$ from $w$.
	\end{dfn}
	
	\begin{exm}
		Let $w = 1123234554$. Then
		\begin{enumerate} \itemsep1pt
			\item $w - 4554 = 112323$,
			\item $w - \{11, 4554\} = ((w - 4554) - 11) = 2323$, and
			\item $w - 3 = 11224554$.
		\end{enumerate}
	\end{exm}

\subsection{Reductions motivated by biology}

Several sources (\cite{hoffman1997}, \cite{prescott2000}, and \cite{chang2005}, for example)  
have observed frequently occurring sequences in the scrambled micronuclear genome of certain ciliate species.
 The sources propose theories that relate the nesting of these sequences in micronuclear DNA to the evolutionary complexity of the species.
 Potentially, the more nested the sequences are, the more mutated, or evolved, the ciliate species may be.
 In the present section we introduce double occurrence words of a specific form to match the observed sequences
and we use these words to introduce the notion of a nesting index of a double occurrence word.
From a biological perspective the nesting index could be seen as a measurement of the evolutionary complexity of a scrambled ciliate genome.
\begin{dfn}
A \emph{return word} is a word of the form
\[
				a_1a_2\cdots a_n a_n \cdots a_2 a_1,  \qquad a_i \in \Sigma \text{ for all } i, \text{ and } a_i \neq a_j \text{ for } i \neq j.
\]
A \emph{repeat word} is a word of the form
\[
				a_1 a_2 \cdots a_n a_1 a_2 \cdots a_n, \qquad a_i \in \Sigma \text{ for all } i, \text{ and } a_i \neq a_j \text{ for } i \neq j.
\]
\end{dfn}

\begin{figure}[h]
	\centering
	\begin{subfigure}{0.45\textwidth}
		\vspace*{2pt}
		\includegraphics[scale = 0.75]{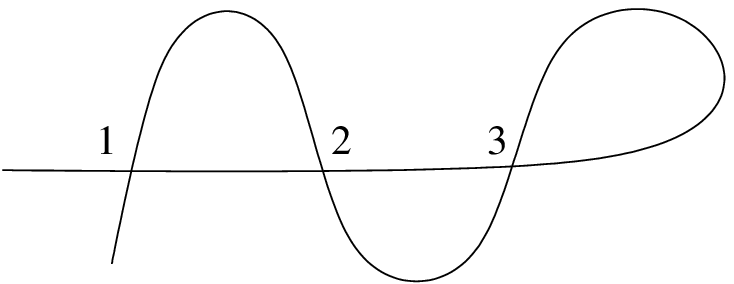}
		\label{fig:TypeA}
		\caption{123321 is a return word}
	\end{subfigure}
	\quad
	\begin{subfigure}{0.45\textwidth}
		\includegraphics[scale = 0.61]{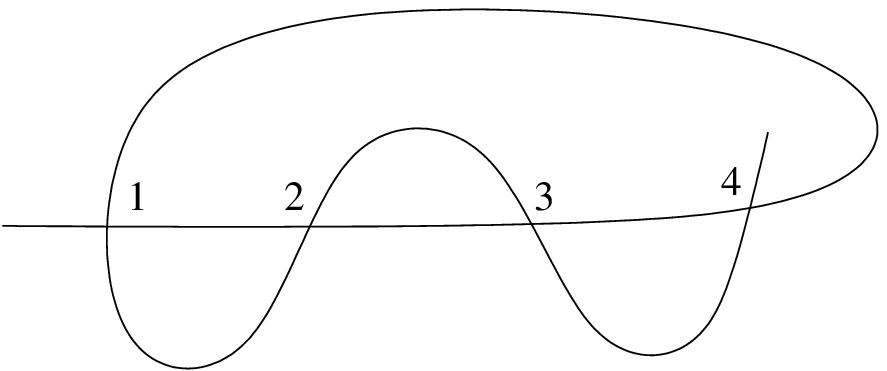}
		\label{fig:TypeB}
		\caption{12341234 is a repeat word}
	\end{subfigure}
	\caption{Assembly graphs of a repeat word and a return word}
\end{figure}

\begin{rem} \label{rem:DOseqs}
All repeat words and return words are double occurrence words.
\end{rem}

\begin{dfn} \label{def:maximal}
Let $\R$ denote the set of all repeat words and return words and let $w$ be a double occurrence word.
Then a word $u$ said to be a \emph{maximal subword} of $w$  with respect to $\R$ if $u \sqsubseteq w$, $u \in \R$, 
and $u \sqsubseteq v \sqsubseteq w$ implies $v \notin \R$ or $u = v$.
\end{dfn}

When we wish to distinguish between repeat words and return words
we sometimes say a maximal return word of $w$ to mean a return word that is a maximal subword of $w$ with respect to $\R$
and similarly for a maximal repeat word of $w$.
Note that the word $aa$ for some $a \in \Sigma$ may be a maximal subword with respect to $\R$ which is both a repeat word and a return word.
In the remainder of the paper, a maximal subword of a word $w$ will mean a maximal subword with respect to $\R$.

\begin{exm}
	Let $w = 1233214545$.
	Then $123321$, $2332$, $33$, and $4545$, are all subwords of $w$ which are repeat or return words.
	$2332$ and $33$ are not maximal subwords because they are subwords of the return word $123321$.
	On the other hand, $123321$ and $4545$ are maximal  subwords of $w$.
\end{exm}

\begin{rem} \label{rem:preandpostnotDO}
 If $s$ is a repeat word or a return word and we write $s = uv$  where $u$ and $v$ are both non-empty, 
then neither $u$ nor $v$ is a double occurrence word.
\end{rem}

Note that if $S$ is a set of double occurrence subwords of $w$,
and the words in $S$ are not pairwise disjoint, 
then $w - S$ may not be defined as it is for disjoint subwords in Definition~\ref{def:set-removal}.
The following lemma and corollary show that
if $\M_w$ is the set of maximal subwords of a double occurrence word $w$,
then $\M_w$ is a set of disjoint subwords of $w$, hence,
$w - \M_w$ is defined.

\begin{lem}
Let $w$ be a double occurrence word with subwords $s_1$ and $s_2$, such that $s_1 \in \R$ and $s_2 \in \R$.
If $s_1 \not\sqsubseteq s_2$ and $s_2 \not\sqsubseteq s_1$, then $s_1$ and $s_2$ are disjoint words.
\end{lem}

\begin{proof}
Assume to the contrary that $s_1$ and $s_2$ have at least one letter in common.  
First, consider the case that there exists a subword separating $s_1$ and $s_2$, that is $w = u_1 s_1 u_2 s_2 u_3$. 
However, since $s_1$ and $s_2$ are double occurrence words (Remark~\ref{rem:DOseqs}),
a separation  would contradict the assumption that $w$ is double occurrence. 
Note that the outcome is the same if we let any combination of $u_1$, $u_2$ and $u_3$ be empty words.

Then suppose the subwords $s_1$ and $s_2$ have an overlap, 
meaning that without loss of generality we can write $s_1 = v_1u$ and $s_2 = uv_2$. 
Since $s_1 \not\sqsubseteq s_2$ and $s_2 \not\sqsubseteq s_1$, it follows that $v_1$ and $v_2$ are non-empty.
However, $u$ can not be a double occurrence word (Remark~\ref{rem:preandpostnotDO}). 
Then there exists a letter $a$ in $u$ such that $a$ has only one occurrence in $u$. 
However, since $s_1$ and $s_2$ are double occurrence words (Remark~\ref{rem:DOseqs}), then $a$ has at least 3 occurrences in $w$. 
This contradicts the fact that $w$ is a double occurrence word.
\end{proof}

Directly from Definition~\ref{def:maximal} we obtain the following corollary.

\begin{crl}
	If $u_1$ and $u_2$ are distinct maximal subwords of a double occurrence word $w$, then $u_1$ and $u_2$ are disjoint words.
\end{crl}

Using the notion of maximal subwords we define two reduction operations on double occurrence words.

\begin{dfn}
	Let $w$ be a double occurrence word.
	We say $w'$ is obtained from $w$ by \emph{reduction operation 1} if
		$w' = w - \{u : u \text{ is a maximal subword of } w\}.$
	We say $w'$ is obtained from $w$ by \emph{reduction operation 2} if for some $a \in \Sigma$, $w' = w - a$.
\end{dfn}

	Figure~\ref{fig:reduction-ops} gives an example of each reduction operation applied to the word $123324564561$.

	\begin{figure}[h]
				\centering
				\begin{subfigure}[h]{0.49\textwidth}
					\centering
					\includegraphics[scale=0.36]{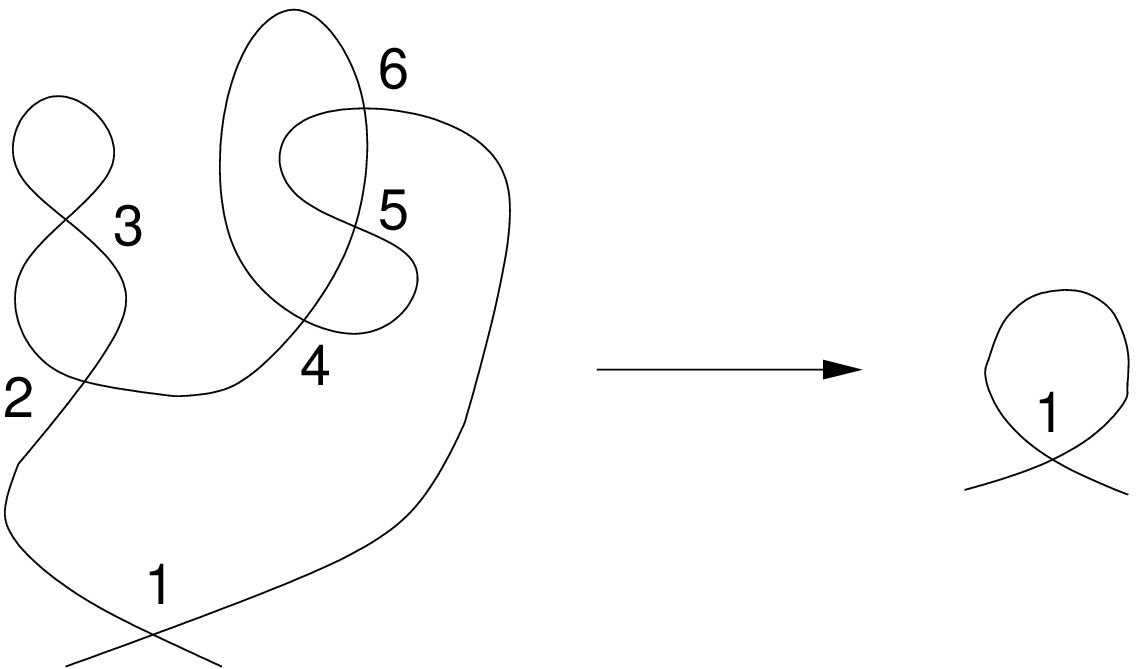}
					\vspace*{10pt}
					\caption{11 obtained from 123324564561}
				\end{subfigure}
				\begin{subfigure}[h]{0.49\textwidth}
					\centering
					\vspace*{11pt}
					\includegraphics[scale=0.36]{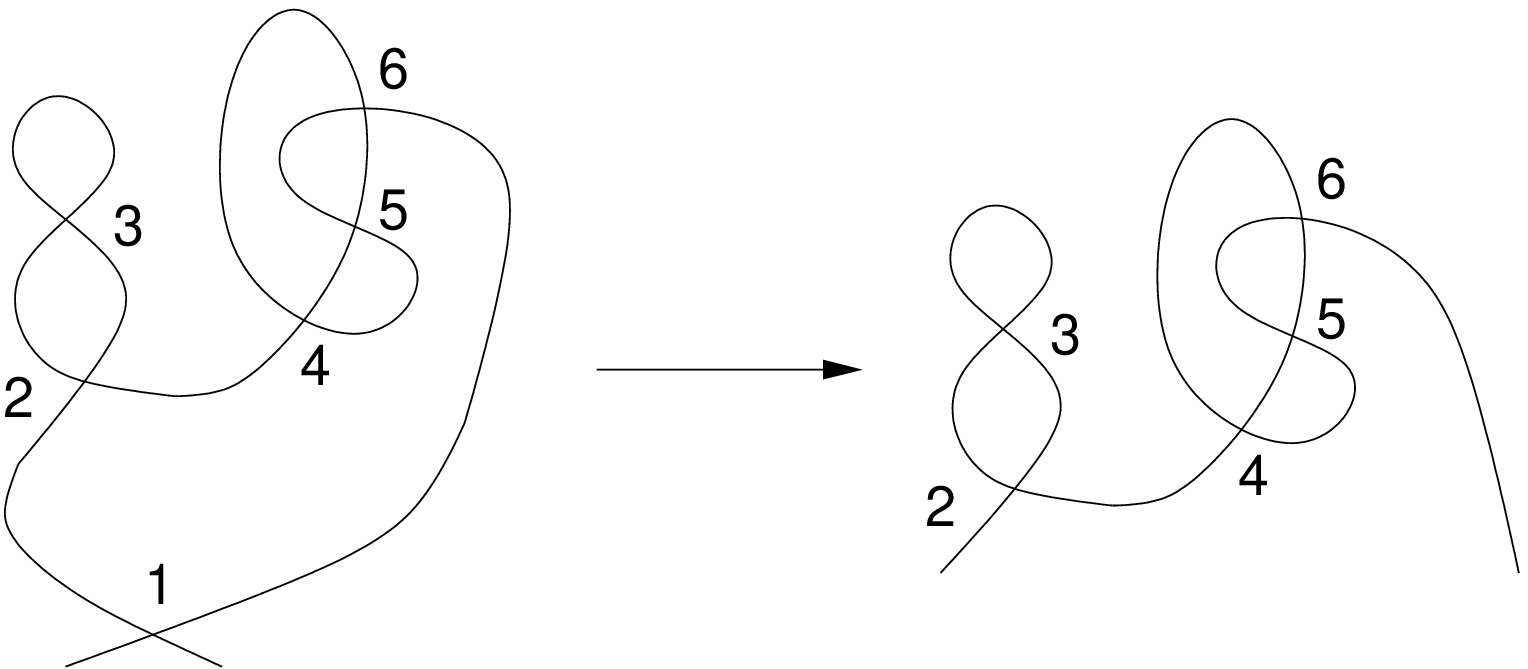}
					\vspace*{-2pt}
					\caption{2332456456 obtained from 123324564561}
				\end{subfigure}
				\caption{Examples of reduction operations 1 (left) and 2 (right)}
				\label{fig:reduction-ops}
	\end{figure}

\begin{dfn} \label{def:reduction}
	A \emph{reduction} of $w$ is a sequence of words $(u_0, u_1, \ldots, u_n)$ in which
	\begin{inparaenum}[(1)] \itemsep1pt
		\item $u_0 = w$,
		\item for $0 \leq k < n$, $u_{k+1}$ is obtained from $u_{k}$ by application of one of the reduction operations, and 
		\item $u_n = \epsilon$.
	\end{inparaenum}
\end{dfn}

	Note that every double occurrence word has at least one reduction 
	(in any case we can remove a letter from $u_i$ to obtain a possible $u_{i+1}$),
	and most double occurrence words, in fact, have many distinct reductions.

\begin{exm}
	Consider $w = 1234554231$.
	Applying reduction operation 1 to $w$ gives $w_1 = 123231$.
	A second application of the reduction operation to $w_1$ gives $11$, and so a third application gives $\epsilon$.
	Then $R_1 = (1234554231, 123231, 11, \epsilon)$ is a reduction of $w$.
	For a second example, if we apply reduction operation 2 to $w$ by removing the letter 3, we get $w_1' = 12455421$.
	Since $w_1'$ is a return word, an application of reduction operation 1 to $w_1'$ gives $\epsilon$.
	Then $R_2 = (1234554231, 12455421, \epsilon)$ is also a reduction of $w$.
	\label{ex:reductions}
\end{exm}

\begin{dfn}
	A double occurrence word $w$ is called \emph{$1$-reducible} if there exists a reduction $(u_0,u_1, \ldots, u_n)$ of $w$
	such that for all $0 \leq i < n$, $u_{i+1}$ is obtained from $u_i$ by application of reduction operation $1$.
\end{dfn}

	In the previous example we saw that $w = 1234554231$ is $1$-reducible by reduction $R_1$.
	In the following section we give a characterization of words which are $1$-reducible.

\begin{dfn} 
	$\NI := \min\{n: (u_0, u_1, \ldots, u_n) \text{ is a reduction of } w\}$ is the \emph{nesting index} of the double occurrence word $w$.
\end{dfn}

	In \cite{angeleska2009} it is shown that two assembly graphs $\Gamma_1$ and $\Gamma_2$ are isomorphic if and only
	if the double occurrence words of $\Gamma_1$ and $\Gamma_2$ are reverse equivalent.
	Note that if $w_1$ and $w_2$ are reverse equivalent, then every repeat (return) word in $w_1$ appears
	as a repeat (return) word in $w_2$.
	Then there is a one-to-one correspondence between reductions of $w_1$ and reductions of $w_2$,
	hence, $\NIin{w_1} = \NIin{w_2}$.
	It follows that the nesting index is an invariant of assembly graphs.

	Note that in Example~\ref{ex:reductions}, the second word in $R_1$ is obtained from $w$ by removing a subword of length $4$.
	In $R_2$ the second word is obtained from $w$ by a letter removal.
	Although we removed less from $w$ in the beginning for $R_2$, 
	the number of reduction operations needed to reduce $w$ to the empty word was less than in $R_1$.
	This example shows that a greedy algorithm based on the number of letters to be removed would be incorrect 
	for the computation of the nesting index.
	The current 
	algorithm\footnote{Implemented in C code, readily available for download at \url{http://knot.math.usf.edu/software/NI/nest_index.c}}
	to compute the nesting index is slightly better than brute force.
	It is unknown whether there exists a more efficient algorithm to compute the nesting index of a double occurrence word.
	
	Using the aforementioned C program we were able to obtain Table~\ref{tab:NI_counts}
	which gives counts on the number of double occurrence words (labeled in ascending order) with a given size and nesting index.
	For words of size $\leq 9$ the counts for all nesting indices are given.
	For words of size $10, 11, 12$, the number of words is quite large 
	and so the computation for all nesting indices would be somewhat time consuming.
	However, the following lemma allows us to more easily compute the counts of words of size $10,11,12$
	and respective nesting indices $8,9,10$.
	
	\begin{lem}
	If $w$ and $w'$ are double occurrence words such that $w' = w - a$ for some letter $a \in \Sigma$,
	then $\NIin{w} \leq \NIin{w'} + 1$.
	\label{lem:NI-bound}
\end{lem}
\begin{proof}
	If $\NIin{w'} = n$, let $(u_0, u_1, \ldots, u_n)$ be a reduction of $w'$.
	Then \linebreak $(w,u_0, u_1, \ldots, u_n)$ is a reduction of $w$ in which $u_0 = w' = w-a$.
	Thus, $\NIin{w}  \leq  n + 1 = \NIin{w'} + 1$.
\end{proof}

\begin{table}[htb]
\centering
\scalebox{0.78}{
\begin{tabular}{|c|rrrrrrrrrr|}
	\hline
		Size	& \multicolumn{10}{|c|}{Nesting Index}			\\
	&1	&2	&3	&4	&5	&6	&7	&8	&9	&10 \\ \hline
	1	&1	&0	&0	&0	&0	&0	&0	&0	&0	&0 \\
	2	&3	&0	&0	&0	&0	&0	&0	&0	&0	&0 \\
	3	&7	&8	&0	&0	&0	&0	&0	&0	&0	&0 \\
	4	&17	&78	&10	&0	&0	&0	&0	&0	&0	&0 \\
	5	&41	&424	&479	&1	&0	&0	&0	&0	&0	&0 \\
	6	&99	&1915	&6248	&2133	&0	&0	&0	&0	&0	&0\\
	7	&239	&7914	&50247	&69879	&6856	&0	&0	&0	&0	&0\\
	8	&577	&31370	&328810	&1004642	&648065	&13561	&0	&0	&0	&0\\
	9	&1393	&122530	&1927900	&10125920	&17081040	&5187788	&12854	&0	&0	&0\\
	10	&	&	&	&	&	&	&	&2019	&0	&0\\
	11	&	&	&	&	&	&	&	&	&4	&0\\
	12	&	&	&	&	&	&	&	&	&	&0 \\ \hline
\end{tabular}
}
\caption{Number of double occurrence words (labeled in ascending order) with a given size and nesting index}
\label{tab:NI_counts}
\end{table}

In the concluding remarks, we use Table~\ref{tab:NI_counts}
to give a conjecture on the minimum number of letters needed to construct a word with nesting index $n \in \mathbb{N}$.

\section{A study on the nesting index}
Chord diagrams and circle graphs  
are useful tools in the study of double occurrence words, for example in \cite{chord}.
In the present section we use chord diagrams and circle graphs as tools to study the nesting index of double occurrence words.
We give a characterization of $1$-reducible words.
This characterization allows us to show that for arbitrary $n\geq 0$
there exists a word with nesting index $n$.
We conclude the section with some open questions involving the nesting index.

\subsection{Nesting index and chord diagram}

	A \emph{chord diagram} is a pictorial representation of a double occurrence word $w$ obtained by arranging the $2n$
	letters of $w$ around the circumference of a circle and then for each letter, joining the two occurrences of the same letter	by a chord of the circle.
	A chord diagram $\C'$ is said to be a \emph{sub-chord diagram} of a chord diagram $\C$ if
	the chords of $\C'$ make up some subset of the chords of $\C$.
	Note that every double occurrence word corresponds to some chord diagram 
	but also that two distinct double occurrence words (possibly in ascending order) 
	may correspond to two chord diagrams which differ only by the labeling of chords.
	Occasionally a \emph{base point} in a chord diagram $\C$
	 is used to point out the first letter of the word that corresponds to $\C$.

\begin{exm}
	Figure~\ref{fig:returnchord} and Figure~\ref{fig:repeatchord} are chord diagram representations of the 
	return word 12344321 and repeat word 12341234, respectively.
\end{exm}

\begin{figure}[hbt]
	\centering
	\begin{subfigure}{0.49\textwidth}
		\centering
		\includegraphics[scale=0.6]{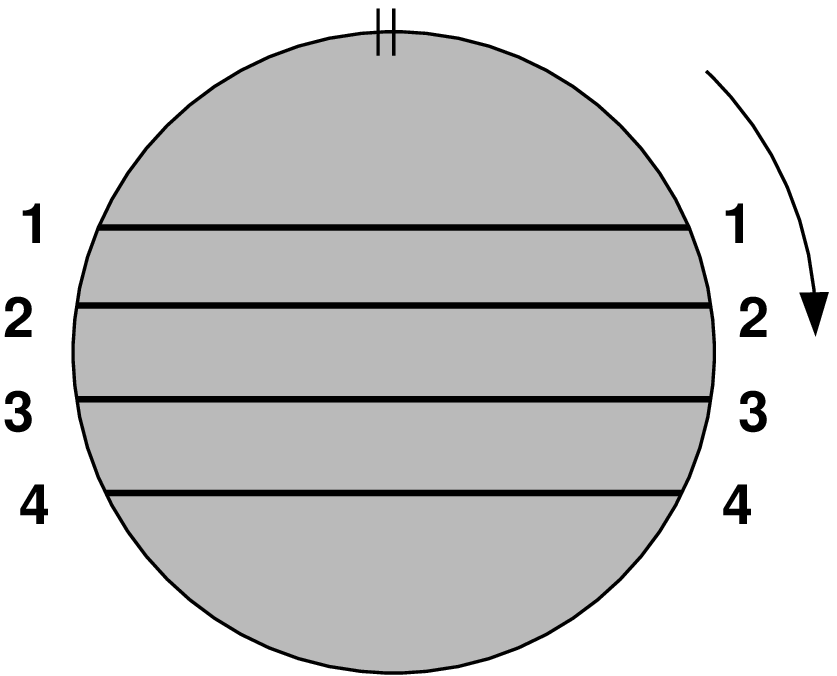}
		\caption{Chord diagram for the return word 12344321}
		\label{fig:returnchord}
	\end{subfigure}
	\begin{subfigure}{0.49\textwidth}
		\centering
		\includegraphics[scale=0.6]{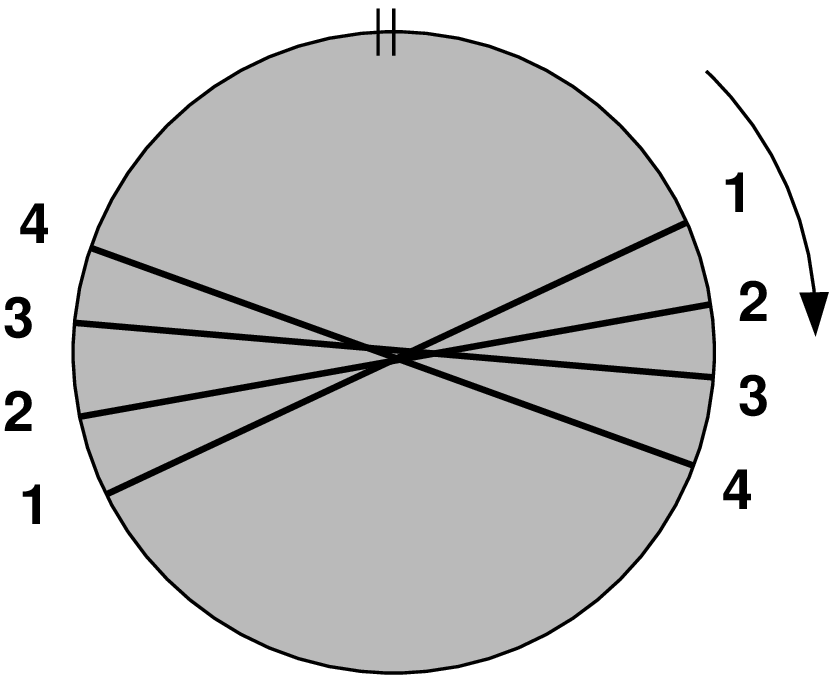}
		\caption{Chord diagram for the repeat word 12341234}
		\label{fig:repeatchord}
	\end{subfigure}
	\caption{Chord diagram representations of double occurrence words}
\end{figure}

\begin{rem} \label{rem:abchords}
	In the chord diagram of any return word no pair of chords intersects.
	In the chord diagram of any repeat word every pair of chords intersects.
\end{rem}

\begin{rem} \label{rem:doswchords}
	If $w$ is a double occurrence word that corresponds to a chord diagram $\C$ 
	and $u \sqsubseteq w$ is also a double occurrence word,
	then the chords in $\C$ associated with $u$ have no intersection 
	with the chords in $\C$  that correspond to the symbols in $w - u$.
\end{rem}

\begin{thm} \label{lem:letter-to-remove}
	Let $w$ be a double occurrence word.
	Then $w$ is $1$-reducible if and only if the chord diagram of $w$ 
	does not contain the chord diagram in Figure~\ref{fig:123213} as a sub-chord diagram.
\end{thm}

	\begin{figure}[h] 
		\centering
		\includegraphics[scale = 0.6]{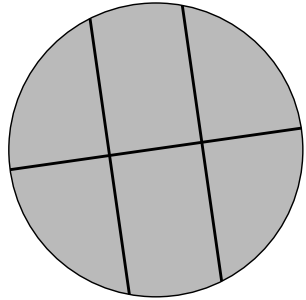}
		\caption{Chord diagram $\C_{1 \times 2}$ associated with double occurrence words 121323, 123213, and 123132}
		\label{fig:123213}	
	\end{figure}

\begin{proof}
	Let us refer the chord diagram of $w$ as $\C$ and the chord diagram in Figure~\ref{fig:123213} as $\C_{1 \times 2}$.
	The proof follows by induction on the size of $w$.
	One can easily verify that all words of size $1$ and $2$ are $1$-reducible 
	and their chord diagrams have less than three chords, hence, do not contain $\C_{1 \times 2}$ as a sub-chord diagram.
	
	Now let $n \geq 3$ and suppose that for arbitrary $k < n$, if $w$ is a word of size $k$, then the theorem holds.
	For the final part of the proof we treat the right and left implications separately.

	($\Rightarrow$):
	Let $w$ be of size $n$ and suppose $w$ is $1$-reducible.
	Let $\M_w$ be the set of maximal subwords of $w$.
	Then $w' = w - \M_w$ is $1$-reducible,
	hence, by induction hypothesis, the chord diagram of $w'$ does not contain $\C_{1 \times 2}$ as a sub-chord diagram.
	By Remark~\ref{rem:doswchords}, the chords in $\C$ associated with the words in $\M_w$,
	have no intersection with the chords in $\C$ associated with $w'$.
	Then if $\C_{1 \times 2}$ is a sub-chord diagram of $\C$, $\C_{1 \times 2}$ must be a sub-chord diagram of
	the chords in $\C$ associated with the words in $\M_w$.
	However, since a pair of chords associated with two distinct double occurrence words in $\M_w$ cannot intersect (Remark~\ref{rem:doswchords}),
	it follows that $\C_{1 \times 2}$ must be a sub-chord diagram of the chords associated with a single word $u \in \M_w$.
	But this cannot be the case by Remark~\ref{rem:abchords}.
	Thus, $\C$ does not contain $\C_{1 \times 2}$ as a sub-chord diagram and so the right implication is proved.
	
	($\Leftarrow$):
	Let $w$ be a word of size $n$ and suppose $\C$ does not contain $\C_{1 \times 2}$ as a sub-chord diagram.
	Let $a \in \Sigma$ and let $\C'$ denote the chord diagram of $w' = w - a$.
	Since $\C'$ does not contain $\C_{1 \times 2}$ as a sub-chord diagram, it follows by induction hypothesis that $w'$ is $1$-reducible.
	Let $\M_{w'}$ denote the set of maximal subwords of $w'$.
	
	We claim that $w$ has a maximal subword.
	If for some $u \in \M_{w'}$, $u$ is a subword of $w$, then we are done.
	Since $a$ has only two occurrences in $w$, it follows that if $|\M_{w'}| \geq 3$,
	then there exists $u \in \M_{w'}$ such that $u \sqsubseteq w$ and we are done.
	Assume $|\M_{w'}| \leq 2$.
	If $\M_{w'} = \{u,v\}$ such that $u$ and $v$ are not subwords of $w$,
	then we can write $u = u_1u_2$ and $v = v_1v_2$ such that $u_1au_2$ and $v_1av_2$ are subwords of $w$.
	Since $u$ and $v$ are not subwords of $w$, we have that $u_1$, $u_2$, $v_1$, and $v_2$ are non-empty.
	Since $u_1$, $u_2$, $v_1$, and $v_2$ are non-empty, it follows that they cannot be double occurrence words (Remark~\ref{rem:preandpostnotDO}),
	hence, the chord for $a$ intersects a chord from $u$ and a chord from $v$.
	Since the chords from $u$ and $v$ do not intersect by Remark~\ref{rem:doswchords}, 
	it follows that $\C_{1 \times 2}$ is a sub-chord diagram of $\C$ which is a contradiction.
	Lastly, we consider $\M_{w'} = \{u\}$ in which $u$ is not a subword of $w$.
	Let us write $u = u_1u_2u_3$ so that $u' = u_1au_2au_3$ is a subword of $w$.
	If $u_2$ is empty, then $aa$ is a subword of $w$ which is maximal or contained in a maximal subword of $w$.
	Assume $u_2$ is non-empty.
	If $u$  is a repeat word, then the chord of $a$ in $w$ must intersect all chords of $u$,
	else, $\C_{1 \times 2}$ is a sub-chord diagram of $\C$.
	Since all of the chords of $u'$ intersect, then the word is a maximal repeat word in $w$.
	Now assume $u = a_1a_2 \cdots a_n a_n \cdots a_2 a_1$ is a return word.
	Then the chord of $a$ can intersect at most one chord from $u$, 
	else, $\C_{1 \times 2}$ is a sub-chord diagram of $\C$.
	Suppose $a$ intersects a chord, say with label $a_i$.
	If $i = n$, then $aa_naa_n$ or $a_naa_na$ is a maximal repeat word in $w$.
	If $i \neq n$, then $a_{i+1}a_{i+2} \cdots a_n a_n \cdots a_{i+2} a_{i+1}$ is a maximal return word in $w$.
	Otherwise, assume $a$ intersects no chords from $u$.
	Then $u'$ is a maximal return word of $w$.
	
	By the above claim, we can apply reduction operation $1$ to $w$ to obtain a word $w'$ of size $< n$.	
	Since $\C$ does not contain $\C_{1 \times 2}$ as a sub-chord diagram, the chord diagram of $w'$ also does not contain $\C_{1 \times 2}$.
	By induction hypothesis, $w'$ is $1$-reducible.
	Thus, $w$ is $1$-reducible.
\end{proof}

The preceding theorem tells us that if $\C_{1 \times 2}$ is a sub-chord diagram of 
$\C$ which corresponds to a double occurrence word $w$,
then in any reduction of $w$ we absolutely must apply reduction operation 2.
What it does not tell us is how many times we must apply reduction operation 2.
The following lemma and theorem aim to do just that.

\begin{lem}
	Let $w$ be a double occurrence word with chord diagram $\C$
	and let $w'$ be the word obtained from $w$ by application of reduction operation 1 with chord diagram $\C'$.
	If $\C_{1 \times 2}$ is a sub-chord diagram of $\C$ and $b$ is a chord in $\C_{1 \times 2}$,
	then $b$ is also a chord in $\C'$.
	\label{lem:chords-stay}
\end{lem}

\begin{proof}
Assume to the contrary that $b$ is not a chord in $\C'$.
Then $b$ must belong to some maximal subword $u$ of $w$.
Since $b$ is a chord in $\C_{1 \times 2}$, $b$ either intersects the other two chords in $\C_{1 \times 2}$,
or $b$ intersects another chord in $\C_{1 \times 2}$ which intersects the third chord in $\C_{1 \times 2}$.
Then by Remark~\ref{rem:doswchords}, since $u$ is a double occurrence word,
we have that the three letters that correspond to the chords in $\C_{1 \times 2}$ are letters in $u$,
hence, $\C_{1 \times 2}$ is a sub-chord diagram of the chords that correspond to $u$.
However, since $u$ is a repeat word or a return word, then by Remark~\ref{rem:abchords}, this cannot be the case.
This gives a contradiction.
\end{proof}

\begin{thm}
	Let $w$ be a double occurrence word with corresponding chord diagram $\C$ and let $2 \leq n \leq m$ be integers.
	If $\C$ contains the chord diagram in Figure~\ref{fig:mBYn} as a sub-chord diagram, then $\NI \geq n+1$.
	\label{thm:mBYn}
\end{thm} 
\begin{figure}[hbt]
	\centering
	\includegraphics[scale=0.5]{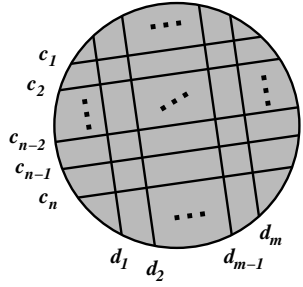}
	\caption{Chord diagram $\C_{n \times m}$}
	\label{fig:mBYn}
\end{figure}

\begin{proof}
	Let $\C_{n \times m}$ denote the chord diagram in Figure~\ref{fig:mBYn}.
	Note that each chord in $\C_{n \times m}$ is a chord in 
	some $\C_{1 \times 2}$ as a sub-chord diagram of $\C_{n \times m}$, hence, as a sub-chord diagram of $\C$.
	Then by Lemma~\ref{lem:chords-stay}, if we apply reduction operation 1 some number of times to $w$ to obtain $w'$,
	then $\C_{n \times m}$ remains a sub-chord diagram of the chord diagram of $w'$.
	Then we must apply reduction operation 2 to remove any letter from $w$ corresponding to some chord in $\C_{n \times m}$.
	Further, note that if we remove a chord from $\C_{n \times m}$ by removing the corresponding letter with reduction operation 2,
	then every chord in the resulting chord diagram $\C_{n \times m}'$  is also a chord in some $\C_{1 \times 2}$ as a sub-chord diagram of $\C_{n \times m}'$.
	Hence, by Lemma~\ref{lem:chords-stay}, we are required to apply reduction operation 2 again.
	This necessity of applying reduction operation 2 continues until one of the following occurs.
	\begin{enumerate}[(i)] \itemsep1pt \parskip0pt \parsep0pt
		\item The letters that correspond to the chords $c_1,\ldots,c_n$ have all been removed by $n$ applications of reduction operation 2,
		\item the letters that correspond to the chords $d_1 \ldots, d_m$ have all been removed by $m$ applications of reduction operation 2, or
		\item the letters that correspond to $m-1$ chords $d_i$ and $n-1$ chords $c_j$ have all been removed by $n+m-2$ applications of reduction operation 2.
	\end{enumerate}
	Since $n \leq m \leq n+m -2$, it follows that we must apply reduction operation 2 a minimum of $n$ times for any reduction of $w$.
	This gives $\NI \geq n$.
	Now since there are still chords left over from $\C_{n \times m}$, we see that $w$ has not been reduced to the empty word 
	and so at least one additional reduction operation is necessary to complete a reduction of $w$.
	Thus, $\NI \geq n+1$.
\end{proof}

\begin{crl}
	For all $n \in \mathbb{N}$, there exists a double occurrence word $w$ with $\NI = n$.
\end{crl}

\begin{proof}
	We have $\NIin{11} = 1$, $\NIin{123231} = 2$ and 
	for $n \geq 3$, by Theorem~\ref{thm:mBYn}, 
	we can take $w$ to be a word corresponding to the chord diagram $\C_{(n-1) \times (n-1)}$ in Figure~\ref{fig:mBYn} to get $\NIin{w} = n$.
\end{proof}

We now introduce some notions to rephrase the characterization of $1$-reducible double occurrence words in terms of its subwords.

\begin{dfn}
	If $w = a_1 a_2 \cdots a_n$ and $u = a_{i_1} a_{i_2} \cdots a_{i_k}$
	such that $i_1, i_2, \ldots, i_k \in \{1,2,\ldots, n\}$ and $i_1 \leq i_2 \leq \cdots \leq i_k$,
	then we say that $u$ is a \emph{sparse subword}  of $w$.
\end{dfn}

\begin{dfn}
	Let $w$ and $w'$ be double occurrence words.
	If there exists a sparse subword $u$ of $w$ such that $w' = u^{asc}$ then we say that $w'$ is \emph{inherent in $w$}.
\end{dfn}

\begin{crl}
	Let $w$ be a double occurrence word.
	Then $w$ is $1$-reducible if and only
	if neither $123213$, $123132$, nor $121323$ is inherent in $w$.
\end{crl}

\begin{proof}
	Since the words $123213$, $123132$, and $121323$ correspond to chord diagram $\C_{1 \times 2}$ in Figure~\ref{fig:123213},
	it follows that one of the words is inherent in $w$ if and only if $\C_{1 \times 2}$ is a sub-chord diagram of the chord diagram for $w$.
	Then by Theorem~\ref{lem:letter-to-remove}, the result follows.
\end{proof}

\subsection{Nesting index and circle graphs}

	A \emph{circle graph} is a graph $G = (V, E)$ obtained from a chord diagram $\C$ in the following way.
	For each chord in $\C$, we designate a vertex $v \in V$ and for distinct $v_1, v_2 \in V$, 
	we have $\{v_1, v_2\} \in E$ if and only if chords $v_1$ and $v_2$ intersect in $\C$.
Some define a circle graph as the intersection graph of a chord diagram.
Figure~\ref{fig:circle-graph-1212} gives the circle graph representation of the double occurrence word 1212.

\begin{figure}[h]
	\centering
	\begin{subfigure}{0.30\textwidth}
		\centering
		\vspace*{5pt}
		\includegraphics[scale=0.41]{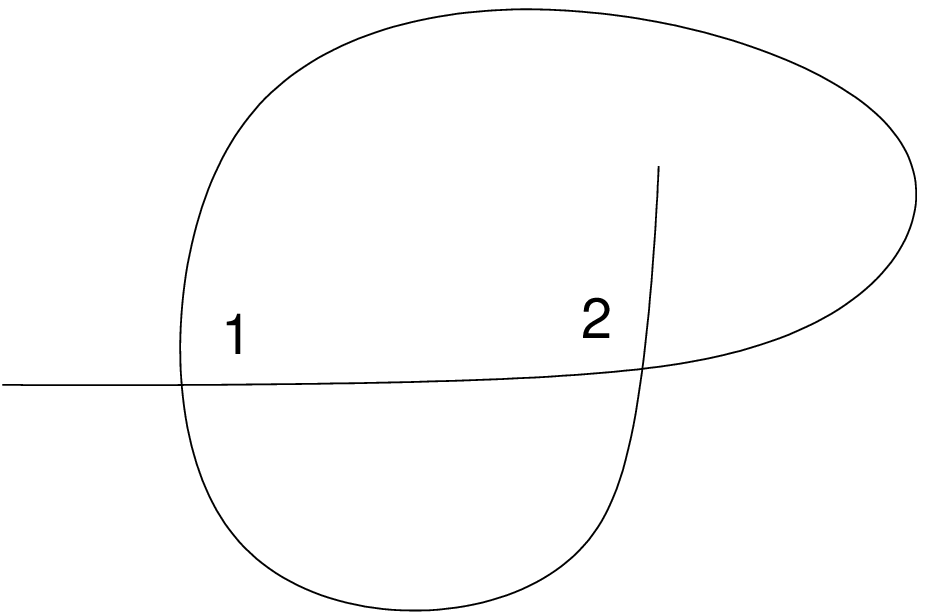}
		\caption{Assembly graph of 1212}
	\end{subfigure}
	\quad
	\begin{subfigure}{0.30\textwidth}
		\centering
		\hspace*{13pt}
		\includegraphics[scale=0.4]{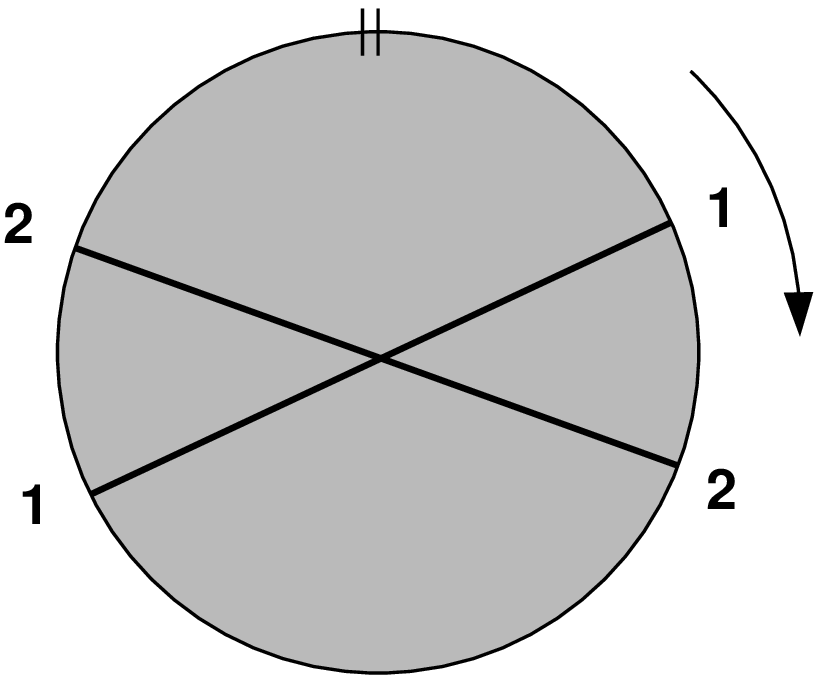}
		\caption{Chord diagram of 1212}
	\end{subfigure}
	\quad
	\begin{subfigure}{0.30\textwidth}
		\centering
		\vspace*{-10pt}
		\includegraphics[scale=0.85]{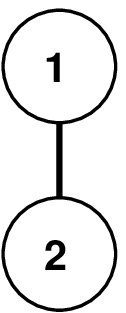}
		\caption{Circle graph of 1212}
		\label{fig:circle-graph-1212}
	\end{subfigure}
	\caption{Various representations of the double occurrence word 1212}
\end{figure}

In the previous subsection 
we found some interesting relationships between the nesting index of a word and the chord diagram of that word.
This prompts the question 
whether any relationships can be found between the nesting index of a double occurrence word and its circle graph.
The following observations, although not a resounding ``no" to the question, 
do show that the nesting index is not an invariant of circle graphs.

Let us consider the words $w_1$ and $w_2$ of size $n \geq 1$ with the following form
\begin{align*}
	w_1 & = 1234 \cdots (2n-1)(2n)(2n-1)(2n) \cdots 3421, \\
	w_2 & = 12123434 \cdots (2n-1)(2n)(2n-1)(2n).
\end{align*}
One can easily verify that for arbitrary $n \geq 1$, we have $\NIin{w_1} = n$ and $\NIin{w_2} = 1$.
Also, Figure~\ref{fig:circle-graphs} shows that the two words correspond to the same circle graph.
Then for arbitrary $n \geq 1$, we can find words of size $n$ 
that correspond to the same circle graph and whose nesting indices differ by $n-1$.

\begin{figure}[h]
	\centering
	\begin{subfigure}{0.4\textwidth}
		\centering
		\vspace*{5pt}
		\includegraphics[scale=0.5]{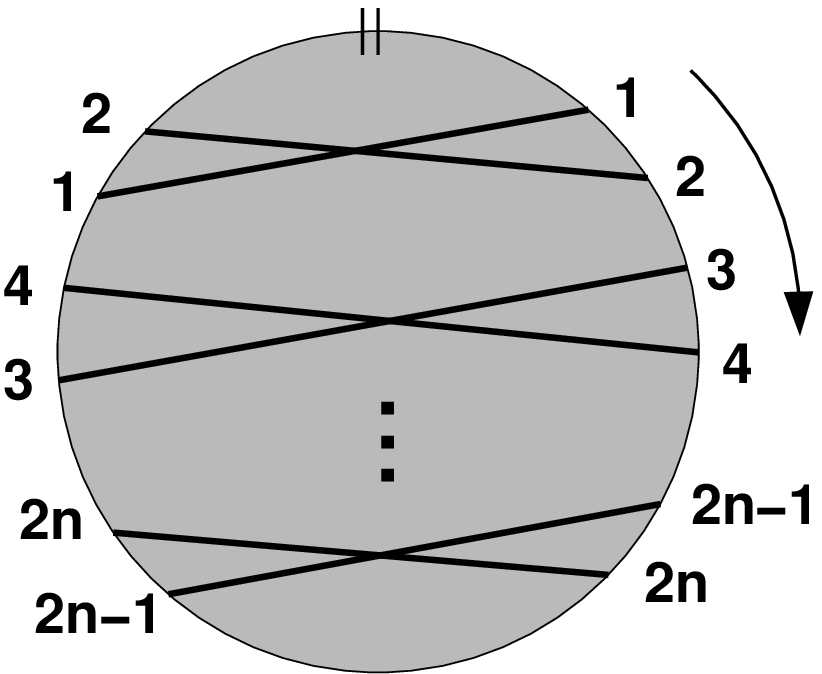}
		\vspace*{4pt}
		\caption{Chord diagram of $w_1$}
	\end{subfigure}
	\quad
	\begin{subfigure}{0.4\textwidth}
		\centering
		\includegraphics[scale=0.5]{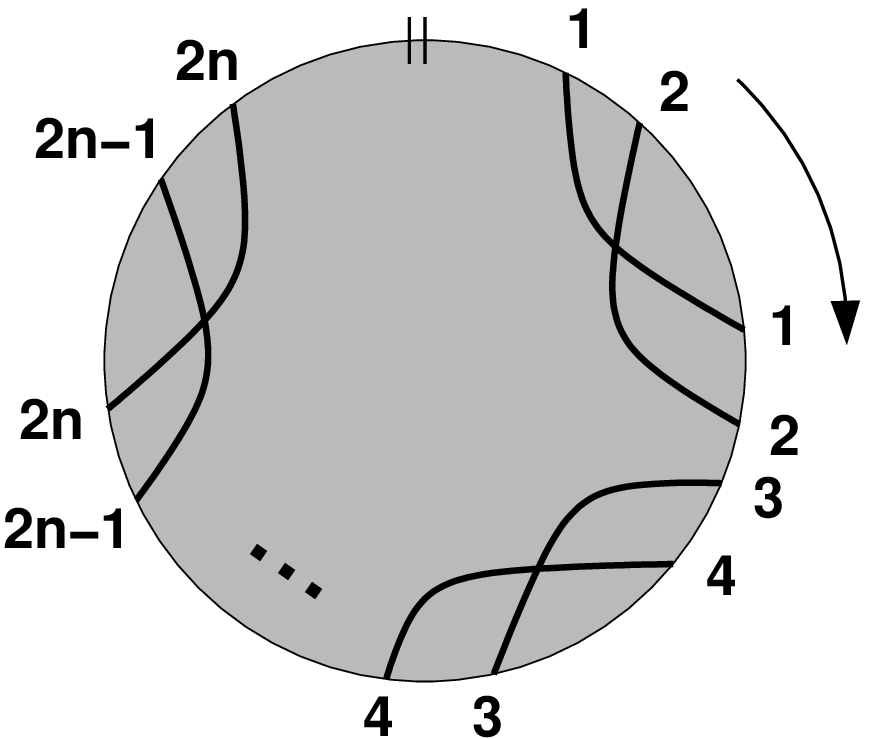}
		\caption{Chord diagram of $w_2$}
	\end{subfigure}
	\begin{subfigure}{0.5\textwidth}
		\centering
			\vspace*{20pt}
			\hspace*{-15pt}
		\includegraphics[scale=0.63]{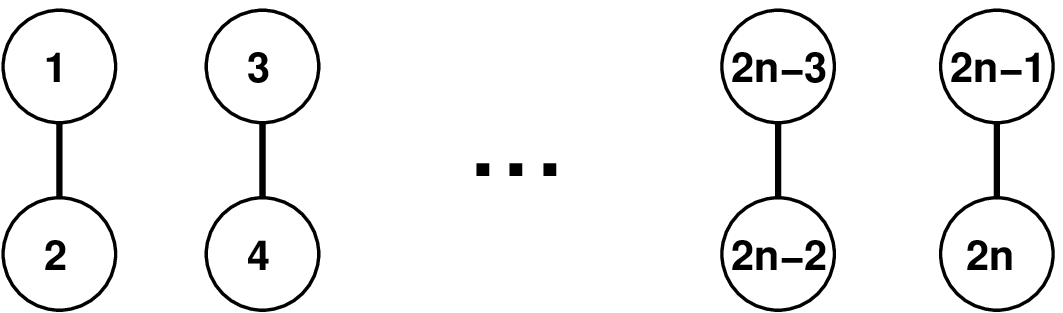}
		\caption{Circle graph of $w_1$ and $w_2$}
	\end{subfigure}
	\caption{Two words that correspond to the same circle graph
	with arbitrarily large differences in nesting indices}
	\label{fig:circle-graphs}
\end{figure}

\subsection{Concluding remarks}
In the previous sections, we introduced the notion of a nesting index as an invariant of assembly graphs.
We gave a characterization of words that are $1$-reducible.
We showed that the nesting index is not an invariant of circle graphs.
Now we conclude the paper with some conjectures and open questions.

The counts in Table~\ref{tab:NI_counts} motivate the following conjecture.
 \emph{Let $n \geq  1$ be an integer and let $s$ be the number of non-zero squares less than $n$.
 Then the number of letters needed to construct a word with nesting index $n$ is $n-s$}.
 
Despite the comments made on the circle graphs in relation to the nesting index,
we still believe there are some interesting questions on the topic.
Given a circle graph and a maximal set of words that realize it,
when do those words have the same nesting index?
Let $S$ be a set of double occurrence words with the same circle graph and same nesting index.
Does there exist an integer $N$, independent of $S$, such that $|S| \leq N$?


\section*{Acknowledgments}
This work has been supported in part by the NSF Grant DMS \#0900671.

The author thanks N. Jonoska, M. Saito and their research group
for their valuable comments and assistance in the work presented here.
For work on related topics the reader is advised to visit the research group's website at \linebreak
\url{http://knot.math.usf.edu}.


\end{document}